\documentclass{article}

\usepackage[UKenglish]{babel}

\usepackage[matrix,arrow,ps]{xy}
\UsePSspecials{dvips}

\usepackage{amsmath,amscd,amsthm, amssymb}

\usepackage{bbm}

\makeatletter
\newtheorem*{rep@theorem}{\rep@title}
\newcommand{\newreptheorem}[2]{%
\newenvironment{rep#1}[1]{%
 \def\rep@title{#2 \ref{##1}}%
 \begin{rep@theorem}}%
 {\end{rep@theorem}}}
\makeatother

\newtheorem{theorem}{Theorem}[section]
\newreptheorem{theorem}{Theorem}
\newtheorem{lemma}[theorem]{Lemma}
\newreptheorem{lemma}{Lemma}

\newreptheorem{corollary}{Corollary}
\newtheorem{proposition}[theorem]{Proposition}
\newreptheorem{proposition}{Proposition}

\newtheorem*{thma}{Theorem A}
\newtheorem*{thmb}{Theorem B}

\theoremstyle{definition}

\theoremstyle{remark}

\begin{document}

\title{Images of word maps in almost simple groups and quasisimple groups}
\author{\scshape{Matthew Levy}\\ \small{Imperial College London}}
\date{}
\maketitle

\begin{abstract}
$\vspace{1mm}$

\begin{minipage}{0.77\textwidth}
It has been shown by Lubotzky in \cite{Lubotzky} that the set of verbal images of a fixed non-abelian finite simple group $G$ is precisely the set of endomorphism invariant subsets of $G$. Here we use his result to determine the verbal images of certain almost simple groups and quasisimple groups.
\end{minipage}

\end{abstract}

$\vspace{1mm}$
\section{Introduction}
Let $w$ be a word in the free group of rank $k$. For any group $G$ we can define a word map $w:G^k\rightarrow G$ and we shall let $w(G)$ denote the set of word values, i.e. $w(G):=\{w(g_1,...,g_k):g_i\in G\}$, we will call this the \textit{verbal image} of $w$ over $G$. There has recently been much interest and progress in the study of verbal images over finite groups though the topic has grown from work first begun by P. Hall, see \cite{Segal} for a modern exposition. 


In \cite{KN} it is shown that for any alternating group Alt$(n)$ with $n\ge 5$ and $n\ne 6$ there exists a word $w$ such that Alt$(n)_w$ consists of the identity and all $3$-cycles. This result also holds for Sym($n$). They also construct words whose image over Alt($n$) is the identity and all $p$-cycles for any prime $3<p<n$ and $n\ge 5$. All these words are explicitly given and they go on to give other explicit examples of words whose verbal image over a given group is a single automorphism class and the identity. Other examples can also be found in \cite{Levy}.

In \cite{Lubotzky} Lubotzky has proved that any automorphism invariant subset that contains the identity of any non-ableian finite simple group can be obtained as the image of a word map in two variables. His proof requires the classification of finite simple groups (CFSG). In particular, it uses a result by Guralnick and Kantor \cite{GuralnickKantor} which asserts that any non-identity element of a finite simple group is part of a generating pair. In this article we will extend some of the arguments used by Lubotzky to make similar statements for certain almost simple groups and quasisimple groups thus some of our results also depend on the classification. Our main results are:

\begin{thma}
The verbal images of S$_n$ are either:
\begin{enumerate}
\item[i)] an Aut(S$_n$)-invariant subset of A$_n$ including the identity or;
\item[ii)] the union of an Aut(S$_n$)-invariant subset of S$_n$ and $C$, where $C$ is the set of all $2$-power elements of S$_n$.
\end{enumerate}
\end{thma}

\begin{thmb}
There exists a constant $C$ with the following property: Let $S$ be a universal quasisimple group with $|S|>C$ and let $A$ be a subset of $S$ such that $e\in A$ and $A$ is closed under the action of the automorphism group of $S$. Then there exists a word $w\in F_2$ such that $w(S)=A$.
\end{thmb}
\section{Almost simple groups}

Let $G$ be an almost simple $G$, i.e. we have $S\leq G\leq $ Aut($S$) where $S$ is the unique minimal normal non-abelian simple characteristic subgroup of $G$. A well-known corollary of the CFSG is that Aut$(S)/S$ is solvable hence so is $G/S$. The group $G$ also has the property that any subgroup not containing $S$ doesn't contain $S$ as a composition factor, that is if $H\leq G$ with $S\not\leq H$ then $S\notin$ Comp($H$), where Comp($H$) is the set of composition factors of $H$. To see this observe that if $H\leq G$ with $S\in$ Comp($H$) then since $H/H\cap S$ is solvable we must have that $S\in$ Comp($H\cap S$) thus $S=H\cap S$ giving $S\leq H$. Now suppose that $G$ $\unlhd$ Aut($S$) then we have Aut($K)\leq$ Aut($G)=$ Aut$(S)$ for any $S\leq K\leq G$. For the remainder of this section we fix an almost simple group $G$ with $S\leq G$ $\unlhd$ Aut($S$) where $S$ is a non-abelian finite simple group. We will combine arguments of Ab\'{e}rt in \cite{Abert} and Lubotzky in \cite{Lubotzky} to prove the following theorem:


\begin{theorem}\label{almostsimple}
Let $A$ be a subset of $S$ such that $e\in  A$ and $A$ is closed under the action of the automorphism group of $G$ where $G$ and $S$ are as above. Then there exists a word $w\in F_2$ such that $w(G)=A$.
\end{theorem}

Lubotzky's proof involves studying subdirect products of simple groups. We would like to generalise some of his argument to work for almost simple groups. This is where we use a result of Ab\'{e}rt where he studied such subdirect products, more specifically he looked at subdirect products of just non-solvable groups. The following lemma is due to Ab\'{e}rt (\cite{Abert}) but has been adapted:

\begin{lemma}\label{lemmaabert}
Consider $P=G_1\times ...\times G_n$ where each $G_i\cong G$. Denote by $S_j$ the minimal normal subgroup of $G_j$ so that each $S_j$ is isomorphic to $S$.
Now let $a_{i,j}\in G_j$ for $1\leq i\leq k$ and $1\leq j\leq n$, and suppose that
\begin{align*}
	<a_{1,j},...,a_{k,j}>&= G_j\text{ for }1\leq j\leq s;\\
	<a_{1,j},...,a_{k,j}>=B_j&\geq S_j\text{ for }s< j\leq n,
\end{align*}
and that for $1\leq j<l\leq n$ the $k$-tuples $(a_{1,j},...,a_{k,j})$ and $(a_{1,l},...,a_{k,l})$ are automorphism independent over $G$. For $1\leq i\leq k$ let
$$
h_i=(a_{i,1},...,a_{i,n})\in G_1\times...\times G_n
$$
and let
$$
H=<h_1,...,h_k>\leq G_1\times...\times G_n.
$$
Then
$$
M=S_1\times ...\times S_n \leq H.
$$
\end{lemma}

Before proceeding with the proof of the lemma we first note that, setting $n=s+t$, the above lemma holds in the case where $t=0$; this is proved in \cite{Abert}. We also recall the following from \cite{Abert} though it is not stated in this way:

\begin{lemma}\label{Abert1}
Let $H\leq P$ be a subgroup containing $M$ such that
\begin{align*}
\pi_j(H)&=G_j\text{ for }1\leq j\leq s;\\
\pi_j(H)&\geq S_j\text{ for }s< j\leq n,
\end{align*}
where $\pi_j$ denotes the projection to the $j$-th component. Let $K$ be a normal subgroup of $H$. Then
$$
K\cap M=\prod_{\pi(K)\neq 1}S_j.
$$
\end{lemma}

\begin{proof}
Since $K\cap M$ is normal in  $M=\prod_{1\leq j\leq n}S_j$ it is the direct product of some of the $S_j$, say 
$$
K\cap M=K_1\times...\times K_n,
$$ 
where $K_i=S_i$ or $1$.

If $\pi_j(K)=1$ then $K\cap S_j=1$ so $K_j=1$.

If $\pi_j(K)\neq 1$ then $\pi_j(K)\unlhd \pi_j(H)$ so $S_j\leq\pi_j(K)$ by minimality of $S_j$. Thus
$$
K\cap M\geq[K,M]\geq[K,S_j]=[\pi_j(K),S_j]=S_j,
$$
so $K_j=S_j$.
\end{proof}

\begin{proof}[Proof of Lemma \ref{lemmaabert}]
Consider the projection to the first $n-1$ coordinates:
$$
f:G_1\times...\times G_s\times B_{s+1}\times...\times B_n\rightarrow G_1\times...\times G_s\times B_{s+1}\times...\times B_{n-1}.
$$
Let $H_1=f(H)$ and let $R=\pi_n(\text{Ker}(f))\unlhd B_n$. By induction on $n$ and the above remark, we have $S_1\times...\times S_{n-1}\leq H_1$ and by minimality of $S$ either $R\geq S_n$ or $R=1$.

We claim that $S_n\leq R$. Assume, for contradiction, that $R=1$. Define the map $\theta:H_1\rightarrow B_n$ by
$$
\theta(g_1,...,g_{n-1})=g_n\text{ if }(g_1,...,g_{n-1},g_n)\in H.
$$
Since $(g_1,...,g_{n-1},g_n)$, $(g_1,...,g_{n-1},g_n')\in H$ implies $g_n^{-1}g_n'\in R$, $\theta$ is a well-defined, surjective homomorphism. We are also using the fact that $<a_{1,n},...,a_{k,n}>=B_n$. Now let $K=$ Ker$(\theta)\unlhd H_1$ and so
$$
H_1/K\cong B_n\unrhd S_n,
$$
which is not solvable and has $S$ as a composition factor. Since $S_1\times...\times S_{n-1}\leq H_1$, by lemma \ref{Abert1}
$$
K\cap M=\prod_{\pi(K)\neq 1} S_j.
$$
If $K\geq M$ then $H_1/K$ would be solvable, a contradiction. So there exists a coorodinate $1\leq l <n$ such that $\pi_l(K)=1$, i.e. $K\leq$ Ker($\pi_l$). Since
$$
H_1/\text{Ker}(\pi_l)\cong B_l\unrhd S_l
$$
is a quotient of $H_1/K$ with $S$ as a composition factor we must have $K=$ Ker($\pi_l$). This means that the function $\alpha: B_l\rightarrow B_n$ given by
$$
\alpha(g_l)=g_n\text{ if }(g_1,...,g_l,...,g_n)\in H
$$
is an isomorphism, i.e. $B_l\cong B_n\cong B$, say. Thus $\alpha(a_{i,l})=a_{i,n}$ for ($1\leq i\leq k$) which in turn implies that the tuples $(a_{1,l},...,a_{k,l})$ and $(a_{1,n},...,a_{k,n})$ are not automorphism independent over $B$ and hence $G$, a contradiction. Hence the claim holds and $1\times...\times1\times S_n\leq H$.

Now let $L=f^{-1}(S_1\times...\times S_{n-1})\leq H$ and denote by $L^{(i)}$ the $i$-th term of the derived series of $L$ and let $r$ be a number such that $L^{(r)}=L^{(r+1)}$. Then $f(L^{(r)})=S_1\times...\times S_{n-1}$ and since $1\times...\times 1\times S_n\leq L$ it follows that $1\times...\times 1\times S_n\leq L^{(r)}$. But $J=\pi_n(L^{(r)})\unlhd B_n$ and $J=J'$ and so $J=S_n$. Finally, this implies that
$$
L^{(r)}=S_1\times...\times S_n\leq H.
$$
\end{proof}

We will now follow Lubotzky's argument making use of the above lemma along the way.

\subsection{Proof of \ref{almostsimple}}
Let $\Omega=\{(a_i,b_i):i=1,...,|G|^2\}$ be the set of all ordered pairs of elements from $G$ such that the first $l$ pairs generate a subgroup of $G$ containing $S$ and the remaining pairs generate proper subgroups of $G$ not containing $S$ and thus don't contain $S$ as a composition factor. Consider the homomorphism from $F_2=<x,y>$, the free group on two letters, to the direct product of $|\Omega|$ copies of $G$, $\phi:F_2\rightarrow\prod_{\Omega}G$ where $\phi=\prod_{\Omega}\phi_i$ and each $\phi_i$ is given by the unique homomorphism from $F_2$ to $G$ sending $x$ to $a_i$ and $y$ to $b_i$. Write $\Omega$ as the disjoint union of $\Omega_1$ and $\Omega_2$ where $\Omega_1$ is the set of the first $l$ `generating' pairs in $\Omega$ and $\Omega_2$ is the set of the remaining `non-generating' pairs and write $G_i=\prod_{\Omega_i}G$ accordingly. Set $r_1=(a_1,...a_l)$ and $r_2=(b_1,...,b_l)$ and consider the subgroup of $G_1$ generated by $r_1$ and $r_2$, $R=<r_1,r_2>$. Then by lemma \ref{lemmaabert}, $R$ contains a subgroup $E_1$ isomorphic to $S^r$ for some $r$ depending on the number of automorphism independent orbits of $\Omega_1$. It turns out that $r=\frac{l}{|\text{Aut}(G)|}$:

To see this note that an element $c=(c_1,...,c_l)$ is in $E_1$ if and only if whenever $\alpha\circ\phi_i=\phi_j$ for some $1\leq i,j\leq l$ and $\alpha\in$ Aut($G$), $\alpha(c_i)=c_j$. Now, the group Aut($G$) acts freely on the pairs in $\Omega_1$ and similarly on the set of homomorphims $\{\phi_i:i:=1,...,l\}$. Indeed, if $\alpha\in$ Aut($G$) and $\alpha\circ\phi_i=\phi_j$ (or equivalently, $(\alpha(a_i),\alpha(b_i))=(a_i,b_i))$ then $\alpha$ is the identity automorphism of $G$. It follows that the $l$ epimorphisms form $r=l/|$Aut$(G)|$ orbits. 

Fix $r$ representatives for the orbits of $\Omega_1$ and denote this set by $\Omega_r$. Now we have to worry about the remaining `non-generating' pairs. Set $h_1=(a_i)$ and $h_2=(b_i)$ and let $H=<h_1,h_2>$. Consider the subgroup $D$ of $H$ whose projection to $G_1$ is $E_1$. We claim that this subgroup is of the form $E_1\times E_2$ where $E_2$ is its projection to $G_2$. Let $K_i$ denote the kernel of the projection from $D$ to $G_i$ for each $i$ so that $D/K_i\cong E_i$. Now $K_1$ is a subgroup of $E_2$ whose projection to every single copy of $G$ in $E_2$ is a proper subgroup of $G$ hence $K_1$ has no composition factor isomorphic to $S$. Since $E_1\cong D/K_1$ is isomorphic to $S^r$ we must have that $E_1=K_2$. This is because $K_2$ is a subgroup of $E_1$ and both are isomorphic to $S^r$ as every one of the $r$ composition factors isomorphic to $S$ should appear in $K_2$ as $E_2\cong D/K_2$ has no composition factor isomorphic to $S$. Thus $D=E_1\times E_2$ where $E_1$ is a diagonally, twisted by automorphisms, embedded subgroup of $G_1$ isomorphic to $S^r$. We have proved the following:

\begin{proposition}\label{propsubdirectalmostsimple}
With the notation as above we have
$$
H\geq D=E_1\times E_2=\prod_{(a,b)\in\Omega_r}D(S,(a,b))\times E_2\leq \prod_{\Omega}G.
$$
Here $D(S,(a,b))$ is a diagonal subgroup of $\prod_{\Omega_{(a,b)}^{\text{Aut}(G)}}G$ isomorphic to $S$ where $\Omega_{(a,b)}^{\text{Aut}(G)}$ is the orbit of the pair $(a,b)\in\Omega$ under Aut($G$). More specifically, $D(S,(a,b))=\{(g^{\phi})_{\phi\in\text{Aut}(G)}:g\in S\}$.
\end{proposition}

We continue with Lubotzky's proof: A result of Guralnick and Kantor (see \cite{GuralnickKantor}) says that for every non-abelian finite simple group $S$ and for every $e\neq a\in S$ there exists $b\in S$ such that $a$ and $b$ generate $S$. Now suppose the set $A'=A\setminus\{e\}< S$ is a union of $k$ Aut($G$)-orbits and  let $T=\{z_1,...,z_k\}$ be a set of representatives for these orbits. It follows that for every $z\in T$ there exists an $1\leq i\leq l$ such that $z=a_i$ with $(a_i,b_i)\in\Omega_1$ and observe that the orbit of $z$ under Aut$(G)$ gives an orbit of $(a_i,b_i)$. It is clear that $k\leq r$.

Define the following element $v=(v_i)\in G_1\times G_2$ by
$$
v_i =
\begin{cases}
a_i & \text{if } a_i\in A'\text{ and }1\leq i\leq l; \\
e & \text{otherwise.}
\end{cases}
$$
It is easy to see that the above element $v$ is an element of $H$. Indeed its projection to $G_2$ is the identity and its projection to $G_1$ lies in $E_1$ by definition and the fact that $A$ is Aut($G$)-invariant. We also observe that the Guralnick-Kantor result ensures that every element of $A$ appears as a coordinate entry of $v$. This means that there exists an element $w\in F_2$ such that its image in $H=\phi(F_2)$ is $\phi(w)=\prod_{\Omega}\phi_i(w)=(w(a_i,b_i))=v$. This word $w$ has the required property, i.e. $w(G)=A$.

\subsection{Verbal images of Symmetric Groups}

As remarked by Lubotzky in \cite{Lubotzky} his proof actually shows, for example, the existence of a word with the following property: 
$$
w(a,b) =
\begin{cases}
v(a,b)& \text{if } <a,b>=G; \\
e & \text{otherwise,}
\end{cases}
$$
where $v$ is any fixed word and $G$ a simple group. In fact, we could choose a different word $v$ for each distinct Aut$(G)$-orbit of the set of pairs of generators. We will make use of this remark as well as the above theorem to classify the verbal images of S$_n$.

For now we will fix $n\geq 5$. As remarked by Kassabov and Nikolov in \cite{KN} if the image of a word $w$ over S$_n$ contains an element outside of A$_n$ then it contains all elements of $2$-power order. To see this note that to contain an element outside A$_n$ the word must have a  free variable with odd exponent sum, i.e. the verbal image of $w$ contains the verbal image of $x^{2m+1}$ for some integer $m$. Let $C$ denote the set of all elements of $2$-power order in S$_n$ including the identity and let $A$ be an Aut(S$_n$)-invariant subset of S$_n$. If $A\subseteq$A$_n$ then the above theorem shows that there exists a word with image precisely $A$. Suppose that $A$ contains an element outside of A$_n$. It follows that $A$ must contain $C$. We claim that any such $A$ can be the image of a word map. Choose a word $v(x_1,x_2)\in F_2$ with the following property:
$$
v(a,b) =
\begin{cases}
a & \text{if } a\in\text{ A}_n,\text{ }(\text{o}(a),2)=1\text{ and }a\in A; \\
a^{-e/e_2}a & \text{if } a\in\text{ A}_n,\text{ }(\text{o}(a),2)\neq 1\text{ and }a\in A; \\
a^{\text{o}(a)_2} & \text{if } a\in\text{S}_n\setminus\text{ A}_n\text{ and }a\in A; \\
1 & \text{otherwise,}
\end{cases}
$$
where $e$ is the exponent of S$_n$, $n_2$ is the 2-part of an integer $n$ and o($g$) is the order of $g\in$ S$_n$. We are also writing $1$ for the identity of S$_n$. Such a word exists by theorem \ref{almostsimple} and noting that any $a\in$ S$_n\setminus($A$_n\cup C$) can generate S$_n$ or A$_n$ since by the Guralnick and Kantor result there exists a $b$ such that $<a^2,b>=$A$_n$ and so $<a,b>=$S$_n$ or A$_n$. Also recall that Aut(A$_n$)=Aut(S$_n$) and note that for any $a\in$ S$_n$, $a^{o(a)_2}\in$ A$_n$. Define $w(x_1,x_2)=x_1^{e/e_2}v$ where $v$ is as above. Then the image of $w$ is precisely $A$ since:
$$
w(a,b) =
\begin{cases}
a & \text{if } a\in\text{ A}_n,\text{ }(\text{o}(a),2)=1\text{ and }a\in A; \\
a & \text{if } a\in\text{ A}_n,\text{ }(\text{o}(a),2)\neq 1\text{ and }a\in A; \\
a' & \text{if } a\in\text{S}_n\setminus\text{ A}_n\text{ and }a\in A; \\
\in C & \text{if }a\notin A.
\end{cases}
$$
where $a'=a^{e/e_2+\text{o}(a)_2}$ is S$_n$-conjugate to $a$ as ($e/e_2+$o$(a)_2$,o($a$))=1.\\

This completes the proof of Theorem A.

\section{Quasisimple groups}

Recall the result of Guralnick and Kantor that says for every finite simple group $\bar{S}$ and for every $e\neq a\in\bar{S}$ there exists $b\in\bar{S}$ such that $a$ and $b$ generate $\bar{S}$. Now for a quasisimple group $S$ we have that for every non-central element $a\in S$ there exists $b\in S$ such that $a$ and $b$ generate $S$. This follows from the easy fact that if a set of elements generate a perfect group modulo its centre then they in fact generate the group, i.e. if $M\leq S$ and $M$Z$(S)=S$ then $M=S$. We will now fix a quasisimple group $S$ that is universal for a simple group $\bar{S}$, i.e. $\bar{S}=S/$Z$(S)$ and $S$ is the unique largest such central extension of $\bar{S}$. The existence of such a $S$ follows from facts about the Schur multiplier of $\bar{S}$ and $S$ is also known as the Schur covering group of $\bar{S}$. In this case Aut$(\bar{S})$ is equal to the group of automorphisms of $\bar{S}$ induced from Aut($S$). In particular, this means that the action of Aut($\bar{S}$) on $\bar{S}$ is equivalent to the action of Aut($S$) on $\bar{S}$. In this section we will use Lubotzky's result to prove the following theorem:

\begin{theorem}\label{quasisimple}
Let $S$ be a universal quasisimple group and let $A$ be a subset of $S$ such that $e\in A$ and $A$ is closed under the action of the automorphism group of $S$. If $S$ has property ($*$) then there exists a word $w\in F_2$ such that $w(S)=A$.
\end{theorem}

\noindent The details of property ($*$) will be given later.

\subsection{Proof of \ref{quasisimple}}
Let $S$ be the universal Schur covering group of a finite simple group $\bar{S}$ so that $\bar{S}=S/$Z$(S)$. Let $\Omega=\{(a_i,b_i):i=1,...,|S|^2\}$ be the set of all ordered pairs of elements from $S$ such that the first $l$ pairs generate $S$ whilst the remaining pairs generate proper subgroups of $S$. Suppose that there are $r$ automorphism independent generating pairs modulo the centre, fixing $r$ representatives for these we may assume that these are the first $r$ pairs of $\Omega$, denote this set by $\Omega_r$. Now consider the homomorphism $\phi:F_2\rightarrow\prod_{\Omega}S$ where, as before, $\phi=\prod_{\Omega}\phi_i$ and each $\phi_i$ is the unique homomorphism sending $x$ to $a_i$ and $y$ to $b_i$ with $F_2=<x,y>$. 

Write $\Omega$ as the disjoint union of $\Omega_1$ and $\Omega_2$ where $\Omega_1$ is the set of the first $l$ pairs from $\Omega$ and $\Omega_2$ are the remaining pairs. Set $S_i=\prod_{\Omega_i}S$, $G=\prod_{\Omega}S$, $Z=\prod_{\Omega}$Z$(S)$ and for a subgroup $K$ of $G$ we will denote by $\bar{K}$ the quotient group $KZ/Z$. For $\Lambda\leq\Omega$ denote by $S_{\Lambda}$ the subgroup of $G$ whose projection to the $\alpha$th component of $G$ is $S$ if $\alpha\in\Lambda$ and is the identity elsewhere so that $S_{\Lambda}\cong S^{|\Lambda|}$. When $\Lambda=\{\alpha\}$ we will just write $S_{\alpha}$ for $S_{\Lambda}$. Let $H=<h_1,h_2>\leq G$ where $h_1=(a_i)$ and $h_2=(b_i)$. Then $\bar{H}\leq G/Z=\prod_{\bar{\Omega}}\bar{S}$. Writing $\bar{\Omega}_r$ for the set $\{(\bar{a_i},\bar{b_i}):i=1,...,r\}$, the image of $\Omega_r$ modulo the centre, it follows from Lubotzky \cite{Lubotzky} that $\bar{H}$ contains a subgroup $\bar{D}$ where
$$
\bar{D}=\prod_{(\bar{a},\bar{b})\in\bar{\Omega}_r}D(\bar{S},(\bar{a},\bar{b})),
$$
where $D(\bar(S),(\bar{a},\bar{b}))$ is as in proposition \ref{propsubdirectalmostsimple}.

We claim that $H$ contains a subgroup $D$ isomorphic to $S_{\Omega_r}$. It is enough to prove the following: if $(a,b)=\alpha\in\Omega_r$ and $\bar{H}$ contains $\bar{S}_{\alpha}$ then $H$ contains a diagonally embedded subgroup isomorphic to $S_{\alpha}$ whose projection to the $\alpha$-th component of $\prod_{\Omega}S$ is $S_{\alpha}$. Let $1\neq\bar{x}$, $\bar{y}\in\bar{S}_{\alpha}$ such that $x$, $y\in S_{\alpha}$ and $H$ contains the vectors $(a_{\beta})$ and $(b_{\beta})$ where $a_{\beta}\in$ Z$(S)$ if $\beta\not\in (a$Z$(S),b$Z$(S))$ and $a_{\beta}\in x$Z$(S)$ if $\beta\in (a$Z$(S),b$Z$(S))$. Similarly $b_{\beta}\in$ Z$(S)$ if $\beta\not\in (a$Z$(S),b$Z$(S))$ and $b_{\beta}\in y$Z$(S)$ if $\beta\in (a$Z$(S),b$Z$(S))$. Then $H$ contains their commutator $[(a_{\beta}),(b_{\beta})]=(c_{\beta})$ where $c_{\beta}=e$ if $\beta\neq\alpha$ and $c_{\beta}=[x,y]$ if $\beta\in(a$Z$(S),b$Z$(S))$. Since $S$ is perfect the claim follows. In summary, we have proved the following:


\begin{proposition}\label{propsubdirectquasisimple}
With the notation as above we have
$$
H\geq D=\prod_{(a,b)\in\Omega_r}Q(S,(a,b))\leq \prod_{\Omega}S,
$$
where $Q(S,(a,b))$ is a diagonal subgroup of $\prod_{(c,d)\in\Omega_{(a,b)}^{\text{Aut}(S)}}\prod_{\Omega_{(c,d)}^{Z(S)}}S$ isomorphic to $S$. Here $\Omega_{(a,b)}^{\text{Aut}(S)}$ is the orbit of the pair $(a,b)\in\Omega$ under Aut($S$) and $\Omega_{(c,d)}^{Z(S)}=\{(cz_1,dz_2):z_i\in Z(S)\}\leq\Omega$. More specifically, $Q(S,(a,b))=\{(g^{\phi},...,g^{\phi})_{\phi\in\text{Aut}(S)}\in \prod_{\phi\in\text{Aut}(S)}\prod_{\Omega_{(a^{\phi},b^{\phi})}^{Z(S)}}S:g\in S\}$.
\end{proposition}

Suppose that $A$ is an Aut($S$)-invariant subset and $A'=A\setminus\{e\}$ has representatives $\{a_1,...,a_k\}$ where $k=k(A)$. Let $\{(x_i,y_i):i=1,..,r\}$ be representatives for Aut$(S)$-independent generating pairs for $S$ that are simultaneously distinct modulo the centre of $S$. Since $\phi(F_2)$ contains a subgroup isomorphic to $S^r$ we require that $r\geq k$. We will then be able to show the existence of a word with the desired property by using the usual argument; define a word $w\in F_2$ such that $w(x_i,y_i)=a_i$ for $i=1,...,k$ and takes the identity elsewhere. We are ready to define property ($*$):\\

\noindent\textit{Property} ($*$): With the notation as above, we say that a group S has property ($*$) if $ r\geq k$ for all $A$. \\

Note that a group $S$ has property ($*$) if and only if the above inequality holds for $A=S$. In the next few sections we will use various results about the probability that a random pair of elements generate a group and bounds on number of conjugacy classes in order to give examples of groups with property ($*$). It is worth pointing out at this stage that the above argument already shows that if $A$ consists of the identity and a single automorphism class then $A$ is the image of a word map.

\subsection{Covers of Alternating Groups}
In this section we continue the above discussion but will focus on the subcase where $\bar{S}$ is an alternating group of degree $n\geq 5$ and $S$ is its Schur covering group. Then for $n\neq 6,7$ we have $S=2.$A$_n$, the double cover of A$_n$, and for $n=6$ or $7$ we have $S=6.$A$_n$. By $\cite{Maroti}$ the number of conjugacy classes, $k(G)$, of a subgroup of S$_n$ with $n\geq 3$ satisfies $k(G)\leq 3^{(n-1)/2}$. It follows that $k(S)\leq 6.3^{(n-1)/2}$. On the other hand by \cite{MarotiTamburini} the probability, $p($A$_n$), that a random pair of elements generate A$_n$ with $n\geq 4$ satisfies $1-1/n-13/n^2<p($A$_n)\leq 1-1/n+2/3n^2$. Hence the number of automorphism independent generating pairs in $S$ $d(S)\geq d(\bar{S})\geq p(\bar{S})|\bar{S}|^2/|\text{Aut}(\bar{S})|=p(\bar{S})|\bar{S}|/|\text{Out}(\bar{S})|$. For $n=5$ we see that $d(2.$A$_5)\geq 9$ whilst $k(2.$A$_5)=k($SL$(2,5))=9$. For $n=6$, as $k(\text{A}_6)=7$, a simple calculation shows that $k(6.$A$_6)\leq 6.7\leq (17/36).(6!/8)\leq d(6.$A$_6)$. By induction, one can show that for $n\geq 7$, $k(S)\leq d(S)$ holds.

\subsection{Groups of Lie type}
Here we consider the case when $S$ is a group of Lie type. As in the case with covers of alternating groups we need good bounds on the number of conjugacy classes and the probability that a random pair of elements generates the group. We will first deal with special linear groups and then look at the general case. For more details about groups of Lie type see, for example, \cite{Atlas} or \cite{GLS}.

\subsubsection{Special Linear Groups}
S will denote SL($n,q$) a special linear group of degree $n$ over a finite field of size $q=p^r$ of characteristic $p$ and so $\bar{S}=$PSL($n,q$). From \cite{GuralnickKantor} we have that the probability that a random pair of elements of $\bar{S}$ generate the group $p(\bar{S})\geq 1-cn^3(\log_{2}q)^2/q^{n-1}$ where $c=36$ for $n\geq 10$ and $c=10^{10}$ if $n\leq 9$. It follows that the number of automorphism independent generating pairs $d(S)\geq d(\bar{S})\geq p(\bar{S})|\bar{S}|^2/|$Aut$(\bar{S})|=p(\bar{S})|\bar{S}|/|$Out$(\bar{S})|$. Now $|$Out$(\bar{S})|\leq 2(n,q-1)\log_{p}q$ and $|\bar{S}|=q^{{n \choose 2}}\prod_{i=2}^{n}(q^i-1)/(n,q-1)$. Putting all this together we have
$$
d(S)\geq \frac{(q^{n-1}-cn^3(\log_{2}q)^2)q^{{n \choose 2}}\prod_{i=2}^{n}(q^i-1)}{2q^{n-1}(n,q-1)^2\log_{p}q}.
$$
On the other hand a result of Fulman and Guralnick, see \cite{FulmanGuralnick}, says that the number of conjugacy classes $k(S)\leq q^{n-1}+3q^{n-2}$. Suppose $n\geq 10$ then we have $d(S)\geq k(S)$ for all $q\geq 4$, for $n\leq 9$ we require that $q\geq 10^{15}$.

In the case where $n=2$ and $q=p$, a prime, we can make the following estimates (see   \cite{KantorLubotzky}):
\begin{align*}
p(S) &\geq  1-1/p-10/p^2;\\
k(S) &\leq  p+4.
\end{align*}
Putting all this together we have $d(S)\geq (p^2-p-10)(p^3-p)/p^2(2,p-1)^2\geq p+4\geq k(S)$ for all $p\geq 5$.

\subsubsection{Other groups of Lie type}
The following estimate for the probability that a pair of elements generates a group of Lie type, $S$, was found by Liebeck and Shalev \cite{LiebeckShalev} answering a question of Kantor and Lubotzky \cite{KantorLubotzky}:
$$
1-p(S)=O(rk(S)^3(\log_2q)^2q^{-rk(S)}).
$$
Here, $rk(S)$ is the untwisted Lie rank of $S$, the rank of the corresponding simple algebraic groups over an algebraic closure of $\mathbb{F}_q$. 
Combining this with a result of Fulman and Guralnick \cite{FulmanGuralnick}, $k(S) \leq 27.2q^{rk(S)}$, we see that for `large enough' $S$, $d(S)\geq k(S)$.\\

Theorem B follows by the CFSG.

\section{Dihedral Groups}

\section{$p$-Groups}

\section{Acknowledgments}
This paper is based on results obtained by the author while studying towards a PhD at Imperial College London under the supervision of Nikolay Nikolov. The author would like to thank Nikolay Nikolov for introducing him to a number of problems associated with word maps and for his support and guidance. The author is also thankful for the financial support provided by EPSRC and Imperial College London during his studies.

\bibliographystyle{plain}
\bibliography{refs}

\end{document}